\documentclass[10pt,leqno,twoside]{amsart}

\usepackage{tikz, subfigure, xcolor} 
\usetikzlibrary{snakes}
\usepackage{amsmath}
\usepackage{amsthm}
\usepackage{amssymb}
\usepackage{amsfonts}
\usepackage{cite}
\usepackage{dsfont}
\usepackage{hyperref}
\usepackage{inputenc}
\usepackage[normalem]{ulem}
\usepackage{doi}

\newtheorem{theorem}{Theorem}[section]
\newtheorem{lemma}[theorem]{Lemma}
\newtheorem{proposition}[theorem]{Proposition}

\theoremstyle{definition}

\newtheorem{remark}[theorem]{Remark}

\numberwithin{equation}{section}

\newcommand{\supp}{\mathrm{supp}}

\renewcommand{\div}{\mathrm{div}\,} 

\providecommand{\norm}[1]{\lVert#1\rVert} 

\DeclareMathOperator{\Ran}{Ran}

\DeclareMathOperator{\tr}{tr}

\newcommand{\R}{\mathbb{R}}

\newcommand{\Z}{\mathbb{Z}}
\newcommand{\N}{\mathbb{N}}

\newcommand{\cL}{{\mathcal L}}

\usepackage{eucal}

\newcommand\restr[2]{{
  \left.\kern-\nulldelimiterspace
  #1 
  \vphantom{\big|}
  \right|_{#2} 
}}
\newcommand{\lpso}{L^p_{\overline{\sigma}}(\Omega)}

\title[Primitive Equations with Linearly Growing Initial
Data]{Primitive Equations with Linearly Growing Initial
Data}

\date{}


\keywords{primitive equations, mild solutions,
	rotating flows\\
	This work was partly supported by the DFG International Research Training Group IRTG 1529 and the JSPS Japanese-German Graduate Externship on Mathematical Fluid Dynamics. 
	The first author has been supported by IRTG 1529 at TU Darmstadt.
	}

\author{Amru Hussein}
\address{Departement of Mathematics,
TU Darmstadt, Schlossgartenstr.~7,
64289 Darmstadt, Germany}
\email{hussein@mathematik.tu-darmstadt.de}

\author{Martin Saal}
\address{Departement of Mathematics,
TU Darmstadt, Schlossgartenstr.~7,
64289 Darmstadt, Germany}
\email{msaal@mathematik.tu-darmstadt.de}

\author{Okihiro Sawada}
\address{Applied Physics Course,
Gifu University, Yanagido~1-1,
501-1193 Gifu, Japan}
\email{okihiro@gifu-u.ac.jp}

\begin{document}

\begin{abstract}
The primitive equations in a $3D$ infinite layer domain
are considered with linearly growing initial data in the
horizontal direction, which illustrates the global
atmospheric rotating or straining flows.
On the boundaries, Dirichlet, Neumann or mixed boundary
conditions are imposed.
The Ornstein-Uhlenbeck type operator appears in the linear
parts, so the semigroup theory is established by Trotter's
arguments due to decomposition of infinitesimal generators.
To obtain smoothing properties of the semigroup,
derivatives of the associated kernel are calculated.
For proving time-local existence and uniqueness of mild
solutions, the adapted Fujita-Kato scheme is used in
certain Sobolev spaces.
\end{abstract}

\maketitle


\section{Introduction}
The primitive equations for the ocean and atmosphere are
considered to be a fundamental model for geophysical flows,
which is derived from the Navier-Stokes or Boussinesq
equations assuming a hydrostatic balance.
Mathematical analysis of the primitive equations has been
commenced by Lions, Teman and Wang
\cite{Lionsetal1992, Lionsetal1992_b, Lionsetal1993}.
For more information on previous works on the primitive
equations, we refer to the articles of Washington and
Parkinson \cite{WashingtonParkinson1986}, Pedlosky
\cite{Pedlosky1987}, Majda \cite{Majda2003} and Vallis
\cite{Vallis2006};
see also the survey by Li and Titi \cite{LiTiti2016} for
recent results and further references.

The $3D$ primitive equations are derived from the
Navier-Stokes equations in domains which are small in the
vertical direction compared to the horizontal ones.
This justifies the assumption of a hydrostatic balance in
the vertical direction.
Although the nonlinear structure of the primitive
equations looks similar to that of the Navier-Stokes
equations, it 
differs due to anisotropic features.
As the Navier-Stokes equations, the primitive equations
describe the velocity $U$ of a fluid and the
pressure $\Pi$.
Putting $U = (V, W)$, where $V = (V_1, V_2)$ denotes the
horizontal components and $W$ stands for the vertical one,
it leads us to
\begin{align*}
\left\{
\begin{array}{rll}
  \partial_t V + V \cdot \nabla_H V + W \cdot \partial_z V
  - \Delta V + \nabla_H \Pi & = f, & \text{ in } \Omega
  \times (0, T),\\
  \Pi_z & =0, &\text{ in } \Omega \times (0, T), \\
  \mathrm{div}_H V + W_z & = 0, & \text{ in } \Omega \times
  (0, T), \\
  V(0) & = V_0, & \text{ in } \Omega.
\end{array}\right.
\end{align*}
Here $\Omega := \R^2 \times (-h, h) \subset \R^3$
for $h > 0$ is an infinite layer domain;
$\nabla_H$ and $\mathrm{div}_H$ denote the horizontal
gradient and the horizontal divergence, respectively;
$\Delta$ stands for the full Laplacian.
It is remarkable that the time-global well-posedness for
the $3D$ primitive equations has been proven by Cao and
Titi \cite{CaoTiti2007} with initial data in $H^1$, while
the question of time-global well-posedness for the
Navier-Stokes equations still constitutes an open problem.
So, in the study of the primitive equations
-- especially the study of admissible initial values --
a legitimate first step would be to ask whether results
known for the Navier-Stokes equations hold also for the
primitive equations.

For the Navier-Stokes equations in the whole space $\R^d$
for $d \geq 2$, Hieber and the third author of this paper
studied in \cite{HieberSawada2005} the particular case of
linearly growing initial data.
Concretely saying, the initial velocity $V_0$ is given by
the form $V_0(x) = v_0(x) - M x$ for $x \in \R^d$,
where $M$ is a constant and trace free matrix and
$v_0 \in L_{\sigma}^p(\R^d)$ for $p \geq d$.
For such an initial velocity the Navier-Stokes equations
are indeed time-local well-posed in the mild sense.
Such initial data occur in several applications.
Typical examples of $M$ with $d=3$ in
\cite{HieberSawada2005, HieberRhandiSawada2007} are
$M = R$, $J$, $S$ and their respective sums with
\begin{align*}
  R = \left(
  \begin{array}{ccc}
    0 & -a & 0 \\ a & 0 & 0 \\ 0 & 0 & 0
  \end{array}
  \right), \quad J = \left(
  \begin{array}{ccc}
    -b & 0 & 0 \\ 0 & -b & 0 \\ 0 & 0 & 2b
  \end{array}
  \right) \quad \text{and} \quad S = \left(
  \begin{array}{ccc}
    c & 0 & 0 \\ 0 & -c & 0 \\ 0 & 0 & 0
  \end{array}
  \right)
\end{align*}
with $a, b, c \in \R$.
Here the matrix $R$ describes a swirl by the Coriolis
force;
$J$ models a drain along the horizontal axis and a jet 
flow in the vertical direction;
$S$ illustrates a model of a straining flow on the surface.
For precise analysis, the reader can find for the case of
pure rotation in e.g. \cite{BabinMahalovNicolaenko2001,
Hishida1999}, for $M=J$ or $S$ in e.g.
\cite{GallayMaekawa2011, Majda1986} and the references
therein. In \cite{CampitiGaldiHieber2014} the global existence 
of strong solutions in the case of a two dimensional exterior domain and a traceless matrix $M$ was 
established for $L^2$-initial data of arbitrary size and 
in \cite{HanShaoWangXu} a weaker version of that 
result was shown recently in the $L^p$-setting ($p\geq 2$) considering only skew-symmetric matrices $M$.

Considering the primitive equations, the setting in the
whole space does not make sense.
However, 
a layer domain $\Omega$ 
imposing Dirichlet,
Neumann or mixed boundary conditions constitutes an
admissible setting 
for the hydrostatic balance assumption.
The investigation in the $L^p$ framework for the primitive
equations in a cylinder $(0, 1)^2 \times (-h, 0) \subset
\R^3$ with some boundary conditions has been started by
Hieber and his collaborators in
\cite{GigaGriesHusseinHieberKashiwabara2016,
HieberHusseinKashiwabara2016, HieberKashiwabara2015}.
They treated 
initial data 
in $H^{2/p, p}$ for
$p \in (1, \infty)$
which carries over to decaying 
initial data 
in $H^{2/p, p}$ on the layer domain.
However, linearly growing data  $V_0(x) = v_0(x) - M x$ 
for $x \in \R^2$ have not been considered,
so far. As an physical application one can think of
\begin{align*}
  M = \left(
  \begin{array}{cc}
    0 & 0 \\  c & 0 
  \end{array}
  \right),
\end{align*}
which produces a horizontal shear flow in the fluid layer. \\
Similar to the case of the Navier-Stokes equations in
the whole space, with $2 \times 2$ matrix $M$ the
substitution
\begin{align*}
  V = v - M x_H, \quad v \in L^p(\Omega)^2, \quad x_H =
  (x, y) \in \R^2
\end{align*}
derives equations involving 
an
Ornstein-Uhlenbeck type
operator in the linear parts, whenever the boundary
conditions are adopted in 
a certain sense.
This operator contains linearly growing coefficients in
the drift terms.
Ornstein-Uhlenbeck type operator generates 
$(C_0)$
semigroup in $L^p$ for $p \in (1, \infty)$, however, this
semigroup is not analytic, in general.
So, it is \textit{a priori} not clear whether it has suitable smoothing
properties or not. 
The handling of the boundary conditions and the linearly
growing coefficients in the drift terms are main technical
issues that one has to overcome for proving eventually the
existence of time-local unique mild solutions as well as
the treatment of the non-linearity.
Although the $3D$ primitive equations 
on bounded cylindrical domains 
are time-global well-posed, it is still open weather the same
results hold with linearly growing data 
on layer domains. In this paper we aim to prove the local in time 
well-posedness of the $3D$ primitive equations. The question of 
time-global solutions will be part of a future work.

Comparing a situation of the Navier-Stokes equations, the
main difficulties in our setting arise from a lack of
smoothing property in the linearized problem.
Although in the whole space the heat kernel is explicitly
given, 
the hydrostatic Stokes operator is
expressed as a perturbation of the Laplacian as shown in
\cite{GigaGriesHusseinHieberKashiwabara2016}.
Hence, even though the structure of non-linearity in the
primitive equations
resembles
that in the
Navier-Stokes equations, 
one has 
to assume some
additional regularity on initial data 
since for primitive equations the non-linearity which is actually 
bi-linear contains derivatives in both arguments.
Therefore, 
it is also more convenient to  
apply the iteration scheme of Fujita-Kato
\cite{FujitaKato1964} rather than that of Kato
\cite{Kato1984} for handling the non-linearity of the
primitive equations.

This paper is organized as follows.
In Section~\ref{sec:Preliminaries} we give preliminaries
for basic setting of function spaces and reformulation for
the problem with linearly growing initial data.
In Section~\ref{sec:layer} we discuss theories of the
Ornstein-Uhlenbeck type operator in an infinite layer
domain, proving smoothing properties of the corresponding
semigroup in Section~\ref{sec:smooth}.
In Section~\ref{sec:mild} the existence of time-local
unique mild solutions is proved.


\subsection*{Acknowledgment}
The authors would like to thank Matthias Hieber for helpful discussions 
on linearly growing data for Navier Stokes equations and Robert Haller-Dintelmann 
for the valuable discussion of Ornstein-Uhlenbeck operators in various domains, in 
particular pointing out the works \cite{HallerWiedl2005} and \cite{Wiedl2006}. 

\section{Preliminaries}\label{sec:Preliminaries}
We consider the $3D$ primitive equations in an infinite
layer domain
\begin{align*}
  \Omega := G \times (-h, h) \subset \R^3 \quad
  \text{with} \quad G := \R^2, \quad h > 0.
\end{align*}
Here the horizontal coordinates are denoted by $x, y \in G$ and
$x_H := (x, y)$, and the vertical one by $z \in (-h, h)$.
For the sake of simplicity, only the velocity and the surface
pressure are considered in this paper, omitting the
temperature, salinity and further quantities which are
incorporated into the full model discussed in
\cite{Lionsetal1992, Lionsetal1992_b, Lionsetal1993}.
The unknowns are the velocity $U$ of the fluid described as
$U = (V, W)$, where the vector $V = (V_1, V_2)$ denotes the horizontal
components and the scalar $W$ stands for of the vertical one,
and the surface pressure written as $\Pi_s$.
So, we consider the equations
\begin{align}\label{eq:primequiv}
  \left\{
  \begin{array}{rll}
    \partial_t V + V \cdot \nabla_H V + W \cdot \partial_z
    V - \Delta V + \nabla_H \Pi_s & = f & \text{ in }
    \Omega \times (0,T), \\
    \mathrm{div}_H \, \overline{V} & = 0 & \text{ in }
    \Omega \times (0,T), \\
    V(0) & = V_0 & \text{ in } \Omega.
  \end{array}\right.
\end{align}
Here $\Pi_s$ can be regarded as a function defined in
$\Omega \times (0, T)$ 
using
$\Pi_s(x, y, z, t) =
\Pi_s(x, y, t)$.
We have used the 
notations 
\begin{align*}
  & \Delta = \partial_x^2 + \partial_y^2+ \partial_z^2,
  \quad \nabla_H  = (\partial_x, \partial_y), \quad
  V \cdot \nabla_H = V_1 \partial_x + V_2 \partial_y, \\
  & \mathrm{div}_H \, V = \partial_x V_1 + \partial_y V_2
  \quad \text{and} \quad \overline{V} = \frac{1}{2 h}
  \int_{-h}^h V(\cdot,\cdot, \xi)d\xi.
\end{align*}

The system is supplemented by the mixed boundary
conditions on
\begin{align*}
  \Gamma_b := G \times \{ -h \} \quad \text{and} \quad
  \Gamma_u := G \times \{ h \}
\end{align*}
denoting 
the bottom and the upper parts of the boundary
$\partial \Omega$, respectively.
Here and hereafter, for justification reducing equations of
disturbance $\eqref{eq:V}$ in below, we impose the
adopted the boundary condition on Dirichlet boundary
condition parts along with linearly growing initial data:
\begin{align}\label{eq:bc}
  V = M x_H \hbox{ on } \Gamma_D \times (0, T), \quad
  \partial_z V = 0 \hbox{ on } \Gamma_N \times (0, T),
  \quad W = 0 \hbox{ on } \partial \Omega \times (0, T).
\end{align}
Here $M \in \R^{2 \times 2}$ is a given constant matrix
with ${\mathrm{tr}} M = 0$;
Dirichlet, Neumann and mixed boundary conditions are
comprised by the notation
\begin{align*}
  \Gamma_D \in \{ \emptyset, \Gamma_u, \Gamma_b,
  \Gamma_u \cup \Gamma_b \} \quad \text{and} \quad
  \Gamma_N := (\Gamma_u \cup \Gamma_b) \setminus \Gamma_D.
\end{align*}
In the  
literatures, several situation of boundary
conditions have also been considered.
For example, in
\cite[Equation (1.37) and (1.37)']{Lionsetal1992} Dirichlet
and mixed boundary conditions are treated respectively,
while in \cite{CaoTiti2007} Neumann boundary conditions are
assumed.

By the boundary conditions on $W$ and $\div U = 0$ 
the vertical component of
the velocity $W = W(V)$ is determined as
\begin{eqnarray*}
  W(x, y, z, t) = - \int_{-h}^z \mathrm{div}_H \,
  V (x, y, \xi, t) d\xi
\end{eqnarray*}
for each $(x, y, z, t) \in \Omega \times (0, T)$;
see e.g. \cite{HieberKashiwabara2015}.

Similarly to the Navier-Stokes equations, one can consider
the solenoidal subspace of $L^p(\Omega)^2$ for
$p \in (1, \infty)$.
The classical Helmholtz projection onto the solenoidal
space $L^p_\sigma(\R^2)$ for $p \in (1, \infty)$ 
is denoted
by $Q_p$ 
here. 
In addition, 
benefits from the approach developed in 
\cite[Section 3 and 4]{HieberKashiwabara2015}
carry over to the present situation.
So, considering 
\begin{align}\label{eq:lpso}
  L^p_{\overline{\sigma}}(\Omega) := \overline{\{v \in
  C^{\infty}(\Omega)^2 : \mathrm{div}_H \, \overline{v}
  = 0 \}}^{\norm{\cdot}_{L^p(\Omega)^2}},
\end{align}
there exists a continuous projection $P_p$, called 
\textit{hydrostatic Helmholtz projection}, from
$L^p(\Omega)^2$ onto $L^p_{\overline{\sigma}}(\Omega)$
which
can be represented by
\begin{align*}
  P_p v = Q_p \overline{v} + \tilde{v},
  \quad \tilde{v} := v - \overline{v}.
\end{align*}
Note that $v = \overline{v}$ holds if and only if
$\tilde v = 0$,
compare also \cite{GigaGriesHusseinHieberKashiwabara2016}.
In particular
$P_p$ annihilates the gradient of the
surface pressure term $\nabla_H \pi_s$.
As in \cite{HieberKashiwabara2015}, we can define the
\textit{hydrostatic Stokes operator} $A_p$ in
$L^p_{\overline{\sigma}}(\Omega)$ by
\begin{align*}
  A_p v := P_p \Delta v, \quad D(A_p) :=  \{ v \in
  H^{2,p}(\Omega)^2 : \restr{\partial_z v}{\Gamma_N} = 0,
  \, \restr{v}{\Gamma_D} = 0 \} \cap
  L^p_{\overline{\sigma}}(\Omega).
\end{align*}

In what follows, we deal with the initial velocity $V_0$
of the form 
\begin{align*}
  V_0(x_H, z) := v_0(x_H, z) - M x_H \quad \hbox{for} \quad 
  x_H \in G, \quad z \in (-h ,h),
\end{align*}
where $v_0 \in L_{\overline{\sigma}}^p(\Omega)$ and
$M \in \R^{2 \times 2}$ with $\tr M = 0$, that is,
\begin{align*}
 M = \left( \begin{array}{cc}
   m_{11} & m_{12} \\ m_{21} & m_{22}
 \end{array} \right) \quad \text{with}
 \quad m_{11}= -m_{22}.
\end{align*} 
The condition $\tr M = 0$ 
implies the
compatibility condition:
\begin{align*}
  \mathrm{div}_H \, \overline{V_0} = 0 \quad \text{in}
  \quad \Omega.
\end{align*}

We now assume that $V$ solves the primitive equations with
initial data $V_0$.
Substituting 
\begin{align*}
  V(x_H, z, t) = v(x_H, z, t) - M x_H \quad \hbox{for} \,
  v(t) \in L_{\overline{\sigma}}^p(\Omega)
\end{align*}
yields for $v$
\begin{align*}
  \partial_t V = \partial_t v, \quad \Delta V = \Delta v,
  \quad \div_H \overline{V} = \div_H \overline{v} \quad
  \text{and} \quad W(V) \partial_z V = W(v) \partial_z v.
\end{align*}
For the last two identities, the assumption $\tr M = 0$
is essential.
Moreover,
\begin{align*}
  V \cdot \nabla_H V = v \cdot \nabla_H v - M x_H \cdot
  \nabla_H v - v \cdot \nabla_H M x_H + M x_H \cdot
  \nabla_H M x_H,
\end{align*}
and the individual terms are more explicitly given by
\begin{align*}
  v \cdot \nabla_H M x_H & = v_1 \begin{bmatrix} m_{11} \\
  m_{21} \end{bmatrix} + v_2 \begin{bmatrix} m_{12} \\
  m_{22} \end{bmatrix} = Mv, \\
  M x_H \cdot \nabla_H v & = (m_{11}x + m_{12}y) \partial_x
  v + (m_{21}x + m_{22}y) \partial_y v, \\
  M x_H \cdot \nabla_H M x_H & = (m_{11}x + m_{12}y)
  \begin{bmatrix} m_{11} \\ m_{21} \end{bmatrix} +
  (m_{21}x + m_{22}y) \begin{bmatrix} m_{12} \\ m_{22}
  \end{bmatrix}.
\end{align*}
Using the assumption $\tr M = 0$, the last quadratic term
simplifies to become
\begin{align*}
  M x_H \cdot \nabla_H M x_H = (m_{11}^2  + m_{12}m_{21})
  x_H = -(\det M) x_H.
\end{align*}
Hence, to absorb this term by the modified pressure, we
put $\tilde \pi_s = \Pi_s - g$ with
$g(x, y) := -\frac{1}{2} \det(M) (x^2+y^2)$.
So, if $(V, \Pi_s)$ solves \eqref{eq:primequiv}, then
$(v, \tilde \pi_s)$ solves the following equations in
$\Omega \times (0, T)$
\begin{align}\label{eq:V}
  \left\{
  \begin{array}{rll}
    \partial_t v + w \partial_z v - \Delta v + v \cdot
    \nabla_H v - M v - M x_H \cdot \nabla_H v + \nabla_H
    \tilde \pi_s & = f, \\
    \mathrm{div}_H \, \overline{v} & = 0
  \end{array}\right.
\end{align}
with initial conditions $v(0)=v_0$ which belongs to a
suitable subspace of $\lpso$.
Since $V(x_H, z) = v(x_H, z) - M x_H$, where $V$ satisfies
the boundary conditions \eqref{eq:bc} one obtains using
$\partial_z M x_H \mid_{\Gamma_N} = 0$, $\partial_z V
\mid_{\Gamma_N} = 0$ and $V \mid_{\Gamma_D} + M x_H
\mid_{\Gamma_D} = 0$ that $v$ satisfies the boundary
conditions
\begin{align}\label{eq:bcequiv}
  v = 0 \hbox{ on } \Gamma_D \times (0, T) \quad
  \hbox{and} \quad \partial_z v = 0  \hbox{ on }
  \Gamma_N \times (0, T).
\end{align}
So, solving system \eqref{eq:V} for a real matrix $M$ with
$\tr M=0$ and boundary conditions \eqref{eq:bcequiv} for
$v(\cdot, t) \in L_{\overline{\sigma}}^p(\Omega)$ and 
initial data $v_0 \in L_{\overline{\sigma}}^p(\Omega)$ is
equivalent to solving the original equations
\eqref{eq:primequiv} with boundary conditions \eqref{eq:bc}
and linearly growing initial data $V_0 = v_0 - Mx_H$.
The linear parts of \eqref{eq:V} are associated with an
operator of Ornstein-Uhlenbeck type which will be
investigated in the next section for general real matrices
$M$, that is not assuming necessarily $\tr M=0$.

The spectrum and the structure of $M$ are relevant for
computing the matrix exponential appearing in the
Kolmogorov kernel given in \eqref{eq:TH} below.
In fact, the eigenvalues of $M$ dominates the dynamics of
solutions to \eqref{eq:V}, directly.
The eigenvalues of $M \in \R^{2\times 2}$ with $\tr M = 0$
are explicitly given as
\begin{align*}
  \lambda_\pm = \pm \sqrt{m_{11}^2 + m_{12} \cdot m_{21}}.
\end{align*} 
Therefore, one can distinguish the following cases:
\begin{itemize}
  \item[(a)] If $m_{11}^2 + m_{12} \cdot m_{21} > 0$, then
  there are two different real eigenvalues, and $M$ is
  similar to a sign-indefinite symmetric matrix, and
  therefore $e^{tM}$ is not exponentially stable.
  \item[(b)] If $m_{11}^2 + m_{12} \cdot m_{21} < 0$, then
  there are two distinct and purely imaginary eigenvalues.
  Hence, $M$ is similar to a skew-symmetric matrix making
  $e^{tM}$ similar to a unitary group and
  $\norm{e^{tM}} < C$ for all $t \in \R$, where $C > 0$
  depends only on $M$.
  \item[(c)] If $m_{11}^2 + m_{12} \cdot m_{21} = 0$, then
  zero is an eigenvalue of algebraic multiplicity two.
  If $M \neq 0$ then $0$ is an eigenvalue of algebraic
  multiplicity one, and hence $M$ is similar to a Jordan
  block with zero on the diagonal, and hence $e^{tM}$ is
  similar to the matrix exponential $\left(
  \begin{array}{cc} 1 & t \\ 0 & 1 \end{array} \right)$.
\end{itemize}
In particular, $\| e^{tM} \| \leq C$ holds with some
$C > 0$ and all $t > 0$ if and only if $M = 0$ or case
$(\mathrm{b})$ holds.
Furthermore, we can take $C = 1$ if $M$ is anti-symmetric,
that is, the pure rotation case.

\begin{remark}
	To absorb the term $M x_H \cdot \nabla_H M x_H$ into the pressure, one needs that it is a gradient field. As described above this holds for traceless matrices. However, this is also true for instance for symmetric matrices, i.e., $m_{12}=m_{21}$, where
	\begin{align*}
	M x_H \cdot \nabla_H M x_H = \nabla \varphi, \quad \varphi(x,y)=
	\frac{1}{2}(m_{11}^2 + m_{12}^2)x^2 + \frac{1}{2}(m_{22}^2 + m_{12}^2)y^2 + m_{12} (m_{11}+ m_{22})xy.
	\end{align*}
\end{remark}


\section{Ornstein-Uhlenbeck operator in a layer}
\label{sec:layer}
We define the Ornstein-Uhlenbeck type operator
$\cL_p$ in $L^p(\Omega)^2$ by
\begin{align*}
  \cL_p v & := \Delta v - (\mathds{1}-P_p) \tr_D 
  (\partial_z v) + M x_H \cdot \nabla_H v - M v, \\
  D(\cL_p) & := \{ v \in H^{2,p}(\Omega)^2 : \partial_z u
  \mid_{\Gamma_N} = 0 \, \hbox{and} \, u \mid_{\Gamma_D}
  = 0, \, M x_H \cdot \nabla_H v \in L^p(\Omega)^2 \}.
\end{align*}
Here we have used
\begin{align*}
  \tr_D (\partial_z v) := \frac{1}{2h} \left( \alpha(b)
  \partial_z v|_{\Gamma_u} - \alpha(u) \partial_z
  v|_{\Gamma_b} \right)
\end{align*}
with $\alpha(a) = 1$ if $\Gamma_a \subset \Gamma_D$ 
for $a \in \{ u, b \}$, and $\alpha(a) = 0$ otherwise.
So, one can consider the following time-evolutionary
ordinary differential equation 
\begin{align*}
  \partial_t v - \cL_p v = f\quad \hbox{for} \quad t>0 
  \quad \hbox{and} \quad v(0) = v_0.
\end{align*}
In fact, by Proposition~\ref{prop:restrictedsemigroup}
below, this problem is well-defined in
$L_{\overline{\sigma}}^p(\Omega)$.
In previous results on 
the bounded cylindrical domain case,
the correction terms $-(\mathds{1}-P_p) \tr_D (\partial_z
v)$ have also been discussed.
In \cite[Section 4]{GigaGriesHusseinHieberKashiwabara2016}
the key idea is to solve the surface pressure terms,
firstly.
This method 
carries over to the case in an infinite layer domain.

To apply known results for Ornstein-Uhlenbeck operators in
the whole space, we decompose $\cL_p$ into the horizontal
parts $\cL_H$ and the vertical one $\cL_z$ as
\begin{align*}
  \cL_H v := \Delta_H v + M x_H \cdot \nabla_H v - M v
  \quad \hbox{and} \quad \cL_z v := \partial_z^2 v -
  (\mathds{1}-P_p) \tr_D (\partial_z v) 
\end{align*}
using anisotropic Sobolev spaces as the domains
\begin{align*}
  D(\cL_H) & := \{ v \in L_z^p H_{xy}^{2, p} : M x_H
  \cdot \nabla_H v \in L^p(\Omega) \}, \\
  D(\cL_z) & := \{ v \in H_{z}^{2, p} L_{xy}^p :
  \partial_z v|_{\Gamma_N} = 0, \, v|_{\Gamma_D} = 0 \}.
\end{align*}
Here, for $r, s \geq 0$ and $p, q \in (1, \infty)$ we
have used the spaces
\begin{eqnarray*}
  H_z^{r, q} H^{s, p}_{xy} := H^{r, q}((-h, h) ;
  H^{s, p}(G))
\end{eqnarray*}
equipped with the norm $\| v \|_{H_z^{r, q} H^{s, p}_{xy}}
:= \| \|v(\cdot, z)\|_{H^{s, p}(G)} \|_{H^{r, q}(-h, h)}$
setting for brevity $H^{0, p} = L^p$. 
Note that 
\begin{align}\label{eq:anisotropicinclusion}
  H^{r+s, p}(\Omega) \subset H_z^{r, p} H^{s, p}_{xy}.
\end{align}
Note that the boundary conditions of $\cL_z$ are
well-defined since the trace on $\Gamma_u$ and $\Gamma_b$
is well-defined in the anisotropic spaces with regularity
in the vertical direction.
Indeed, by Sobolev's embedding we see that
\begin{align*}
  |v(x, y, \pm h)| \leq C \| v(x, y, \cdot) \|_{H^1_z}
  \quad \text{and} \quad |\partial_z v(x, y, \pm h)| \leq
  C \| v(x, y, \cdot) \|_{H^2_z}
\end{align*}
for almost every $x, y \in G$.
Taking $L^p_{xy}$ norms into above, we thus have
\begin{align*}
  \| v(\cdot, \cdot, \pm h) \|_{L^p_{xy}} \leq C
  \| v \|_{H^{1,p}_z L^p_{xy}} \quad \text{and} \quad
  \| \partial_z v(\cdot,\cdot,\pm h) \|_{L^p_{xy}} \leq
  C \| v \|_{H^{2, p}_z L^p_{xy}}.
\end{align*}

The operator $\cL_H$ in $L^p(\R^2)^2$ has been studied in
\cite{HieberSawada2005} drawing back its main properties
to the classical Ornstein-Uhlenbeck operator defined by
$(\cL_H + M)v$ in $L^p(\R^2)^2$ studied extensively in
e.g. \cite{HallerWiedl2005, Metafune2001,
MetafunePruessetall2002, Wiedl2006}.
Consider as in \cite[Lemma 3.3]{HieberSawada2005} for
$L^p(\R^2)^2$, the semigroup $T_H(t) = e^{t \cL_H}$ in
$L^p(\Omega)^2$ for $p \in (1, \infty)$ defined by
\begin{align}\label{eq:TH}
  (T_H(t) \psi)(x_H, z) &
  := \frac{1}{4 \pi (\det Q_t)^{1/2}} e^{-tM} \int_{\R^2}
  \psi(e^{tM} x_H - x'_H, z) e^{-\frac{1}{4} \langle
  Q_t^{-1} x_H', x_H' \rangle} dx_H'
\end{align}
for $t > 0$ and $(x_H, z) \in \Omega$, where $\psi \in L^p
(\Omega)^2$ and $Q_t := \int_0^t e^{s M} e^{s M^T} ds$.
Define the associated kernel $k_t$ by $k_t(x_H) :=
\frac{1}{4\pi (\det Q_t)^{1/2}} e^{-tM} e^{-\tfrac{1}{4}
\langle Q_t^{-1} x_H, x_H \rangle}$, so we can write
(\ref{eq:TH}) as
\begin{align*}
  (T_H(t) \psi)(x_H, z) = (k_t \ast_H \psi)(e^{tM} x_H, z)
  \quad \text{for} \quad t > 0,
\end{align*}
where $\ast_H$ denotes by the convolution
with respect to the $x_H$ variables.

Note that there is a constant $C > 0$ independent of $t$
such that
\begin{align*}
  t^2 \leq \det Q_t \quad \hbox{and} \quad \| Q_t^{-1} \|
  \leq C t^{-1} \quad \text{for} \quad t > 0,
\end{align*}
see \cite[Equation~(3.5)]{HieberSawada2005}.
Let us consider the more general $d$-dimensional case.
Since the matrix $M+M^T$ is symmetric, we can write it as
$A D A^{-1}$ with a diagonal matrix $D$ and a regular
matrix $A$.
Hence, we obtain
\begin{align*}
  \det Q_t = \det \int_0^t e^{s A D A^{-1}} ds
  = \det (A \int_0^t e^{s D} ds A^{-1})
  = \det \int_0^t e^{sD} ds
  = t^k \prod_{i = 1}^{d-k}
  \frac{e^{\lambda_i t}-1}{\lambda_i}.
\end{align*}
Here $k$ is the multiplicity of the eigenvalue 0, and the
$\lambda_i \in \R \setminus \{ 0 \}$ are 
non-zero
eigenvalues.
With $e^{\lambda_i t}-1 \geq \lambda_i t$ for $t > 0$,
we 
obtain $t^d \leq \det Q_t$.
Similarly,
\begin{align*}
  Q_t^{-1} = \left(A \int_0^t e^{s D} ds A^{-1}\right)^{-1}
  = A \left( \begin{array}{ccc}
    t & \cdots & 0 \\ \vdots & \ddots & \vdots \\ 0 &
    \cdots & \frac{e^{t\lambda_{d-k}}-1 }{\lambda_{d-k}}
  \end{array} \right)^{-1} A^{-1}.
\end{align*}
It follows from $e^{\lambda_i t}-1 \geq \lambda_i t$ again
that 
\begin{align*}
  \| Q_t^{-1} \| \leq \|A\| \|A^{-1}\|
  \max_{1 \leq i \leq d-k} \{ t^{-1},
  \frac{\lambda_{i}}{e^{t\lambda_{i}}-1 } \}
  \leq \|A\| \|A^{-1} \| \cdot t^{-1}.
\end{align*}

The operator $\cL^0_z$ defined by 
\begin{align*}
  \cL^0_z \psi = \partial_z^2 \psi, \quad
  D(\cL^0_z) = D(\cL_z)
\end{align*}
is certainly the generator of a bounded analytic semigroup
in $L^p(\Omega)^2$.
In $\R^2 \times S^1$ with $S^1 := \{ z \in [-h,h] : -h
\sim h \}$, a semigroup can be defined explicitly by means
of Fourier series 
\begin{align}\label{eq:Tz}
  (T^0_{z, h}(t) \psi)(x_H,z) := \sum_{k \in \Z} e^{-tk^2}
  \hat{\psi}_k(x_H) e^{ik\frac{\pi}{h}z}
\end{align}
for $t > 0$ and $(x_H, z) \in \Omega$, where
\begin{align}\label{eq:FourierSeries}
  \hat{\psi}_k(x_H) := \frac{1}{2h} \int_{-h}^h
  e^{-ik\frac{\pi}{h}z} \psi(x_H, z) dz
\end{align}
for $t > 0$ and $(x_H, z) \in \Omega$.
We may arrive at the Dirichlet, Neumann and mixed
boundary conditions, taking odd parts of $T^0_{z, 2h}$,
even parts of $T^0_{z, 2h}$, and odd and even part of
$T^0_{z,4h}$, respectively.

We treat $\cL_z = \cL_z^0 + B_p$, which can be regarded as
a relatively bounded perturbation from $\cL^0_z$ by
\begin{align*}
  B_p \colon H_{z}^{1+\frac{1}{p}, p}L_{xy}^p(\Omega)^2 
  \rightarrow L^p(\Omega)^2, \quad
  B_p v := (\mathds{1}-P_p)(\tr_D \partial_z v).
\end{align*}
One sees that the 
operator $B_p$ is bounded by the trace 
theorem
applied to the vertical direction;
see e.g. \cite[Theorem~2.7.2]{Triebel} for half spaces,
which carries over to the situation considered here
localizing functions around $\Gamma_u$ and $\Gamma_b$ by
cut-off functions, 
and
taking $L^p_{xy}$ norm.
In particular, $C^1(-h, h) \hookrightarrow
H_{z}^{1+\frac{1}{p}, p}$ for $p \in (1, \infty)$.
Also, $D(\cL_z^0) \subset D(B_p)$ holds.
Therefore, by interpolation inequality and
$H_z^{1+\frac{1}{p}, p} = [L_z^{p}, H_z^{2,p}]_{\frac{1}{2}
+\frac{1}{2p}}$, where $[\cdot, \cdot]_{\theta}$ denotes
the complex interpolation functor,
and
by Young's inequality 
\begin{align*}
  \| B_p v \|_{L^p(\Omega)^2} & \leq C \| v \|_{H_z^{1 +
  \frac{1}{p}, p}L_{xy}^p} \\
  & \leq C \| v \|_{L_z^p L_{xy}^p}^{\frac{1}{2} -
  \frac{1}{2p}} \| v \|_{H_z^{2, p} L_{xy}^p}^{\frac{1}{2}
  + \frac{1}{2p}} \\
  & \leq \varepsilon \, C \| v \|_{H_z^{2, p} L^p_{xy}}
  + \frac{C}{\varepsilon} \| v \|_{L^p(\Omega)^2}
\end{align*}
for $v \in D(\cL_z^0)$ and $\varepsilon >0$, where
$C > 0$ is some constant depending only on $p$ and $h$.
Therefore, $\cL_z$ with $D(\cL_z) = D(\cL_z^0)$ is the
generator of an analytic semigroup $T_z(t) = e^{t \cL_z}$
in $L^p(\Omega)$ for $p \in (1, \infty)$ by e.g.
\cite[Theorem~3.2.1]{Pazy}.
Note that for the case of $\Gamma_D = \emptyset$, i.e.
$\Gamma_N = \Gamma_b \cup \Gamma_u$, we see that
$\cL_z = \cL^0_z$ and the semigroup can be given explicitly
as even parts of $T^0_{z, 2h}$.

\begin{lemma}\label{lemma:commuting}\
\begin{itemize}
  \item[(a)] The operators $T_H(t)$ and $T_z(t)$ define
  $(C_0)$-semigroups in $L^p(\Omega)^2$ for
  $p \in (1, \infty)$ the infinitesimal generators of
  which are $\cL_H$ and $\cL_z$, respectively.
  \item[(b)] They commute, that is,
  \begin{align*}
    T_z(s) T_H(t) = T_H(t) T_z(s) \quad \text{for} \quad
    s, t > 0.
  \end{align*}
  \item[(c)] $T_H(t)$ and $T_z(t)$ restrict to
  $(C_0)$-semigroups in 
  the complementary spaces 
  $L_{\overline{\sigma}}^p(\Omega)$
  and $\nabla_H (W^{1, p}(\R^2))$.
\end{itemize}
\end{lemma}

\begin{proof}
The operator $\cL_z$ is by construction the infinitesimal
generator of $T_z(t)$.
For $\cL_H$, one can verify that the proof of the
corresponding statement on $L^p(\R^2)$ for the generator;
see \cite[Proposition 3.2]{Metafune2001} as well as the
characterization of the domain
\cite[Theorem 4.1]{MetafunePruessetall2002}
which carry over one-to-one to the present situation, and 
one
therefore concludes that $\cL_H$ is indeed the generator
of $T_H$.

In the case of $\Gamma_D = \emptyset$, both semigroups
$T_H(t)$ and $T_z(t)$ are given explicitly by
\eqref{eq:TH} and \eqref{eq:Tz}, respectively.
So, it is straight forward 
to verify directly that they commute
interchanging the order of the integration with respect to
$dx_H$ and $dz$.

Considering the case $\Gamma_D \neq \emptyset$, it suffices
to show that the 
resolvents of
$\cL_H$ and $\cL_z$ commute by
Trotter's approximation formula \cite[Theorem~3.4.4]{Pazy}.
We now prove that their resolvents commute or, equivalently
$\rho(\cL_H) \cap \rho(\cL_z) \neq \emptyset$, that is,
for $\lambda \in \rho(\cL_H)$
\begin{align}\label{eq:commuting_resolvents}
  (\cL_H -\lambda)^{-1} D(\cL_z) \subset D(\cL_z) \quad
  \hbox{and} \quad (\cL_H -\lambda)^{-1} \cL_z v = \cL_z
  (\cL_H -\lambda)^{-1} v
\end{align}
for $v \in D(\cL_z)$;
see e.g. \cite[Section~4.2]{Arendt2004}.
Here, the resolvent of $\cL_H$ can be given explicitly by
the Laplace transform 
\begin{align}\label{eq:LapalceTransform}
  (\cL_H -\lambda)^{-1} v = \int_0^{\infty} e^{-\lambda t}
  T_H(t) v dt \quad \text{for} \quad \lambda > 0,
\end{align}
since $\cL_H$ generates a quasi-contractive
$(C_0)$-semigroup in $L^p(\Omega)^2$.
In \cite[Lemma~3.1]{Metafune2001} $\cL_H^{\prime}v = \cL_H v
+ Mx_H$ is considered, 
which is a given as integral kernel only in $x_H$ direction,
we thus have
\begin{align*}
  T_H(t) \partial_z^2 v = \partial_z^2 T_H(t) v,
  \quad v \in D(\cL_z), \quad t \geq 0,
\end{align*}
since by the arguments of a relatively bounded perturbation 
one has
$D(\cL_z) = D(\cL^0_z)$.
Using \eqref{eq:LapalceTransform}, we see
\begin{align*}
  (\cL_H -\lambda)^{-1} \partial_z^2 v = \partial_z^2
  (\cL_H  -\lambda)^{-1} v \in L^{p}(\Omega)^2 \quad
  \text{for} \quad v \in D(\cL_z).
\end{align*}
Therefore, it holds that
\begin{align*}
  (\cL_H -\lambda)^{-1} D(\cL_z) \subset D(\cL_z).
\end{align*}

We calculate further the horizontal derivatives of
$v(e^{tM} x_H)$.
To shorten the notation, we often omit $z$.
By
the
chain rule we get
\begin{align}\label{eq:partialderivatives1}
  \partial_x [v_1(e^{tM} x_H)] = \langle \partial_x
  e^{tM} x_H, (\nabla v_1)(e^{tM} x_H) \rangle = \langle 
  e^{tM} \left( \begin{array}{c} 1 \\ 0 \end{array}
  \right), (\nabla v_1)(e^{tM}x_H) \rangle
\end{align}
and 
\begin{align}\label{eq:partialderivatives2}
  \partial_y [v_1(e^{tM}x_H)] = \langle e^{tM} \left(
  \begin{array}{c} 0 \\ 1 \end{array} \right),
  (\nabla v_1) (e^{tM} x_H) \rangle
\end{align}
for a scalar function $v_1$, where $\langle\cdot,\cdot\rangle$ denotes the inner product in $\R^2$.
So, we see 
that
\begin{align}\label{eq:partialderivatives3}
  \nabla_H [v_1(e^{tM}x_H)] = (e^{tM})^T (\nabla_H v_1)
  (e^{tM} x_H).
\end{align}
The same holds for $v_2$.
Let $\nabla_H v := (\nabla_H v_1 \;\; \nabla_H v_2)^T$
for $v = (v_1, v_2)$, so we write
\begin{align*}
  \nabla_H [v (e^{tM} x_H)] = [(e^{tM})^T 
  (\nabla v_1 \;\; \nabla v_2) (e^{tM} x_H)]^T =
  (\nabla_H v) (e^{tM} x_H) e^{tM}.
\end{align*}
For checking that $T_H(\cdot)$ restricts to a semigroup on
$L_{\overline{\sigma}}^p(\Omega)$, we now compute
\begin{align*}
  \mathrm{div}_H \, \overline{(T_H(t)v)}(x_H)
  & = \frac{1}{4 \pi (\det Q_t)^{\tfrac{1}{2}}}
  \mathrm{div}_H \, e^{-tM} \int_{\R^2} \overline{v}
  (e^{tM}x_H-x'_H) e^{-\tfrac{1}{4} \langle Q_t^{-1}x_H',
  x_H'\rangle} dx_H' \\
  & = \frac{1}{4 \pi (\det Q_t)^{\tfrac{1}{2}}} \tr
  \nabla_H \left( e^{-tM} \int_{\R^2} \overline{v} (e^{tM}
  x_H - x'_H) e^{-\tfrac{1}{4} \langle Q_t^{-1}x_H', x_H'
  \rangle} dx_H'\right) \\
  & = \frac{1}{4 \pi (\det Q_t)^{\tfrac{1}{2}}} \tr \left(
  e^{-tM} \int_{\R^2} (\nabla_H \overline{v}) (e^{tM} x_H
  - x'_H) e^{-\tfrac{1}{4} \langle Q_t^{-1} x_H', x_H'
  \rangle} dx_H' e^{tM} \right) \\
  & = \frac{1}{4 \pi (\det Q_t)^{\tfrac{1}{2}}} \int_{\R^2}
  (\mathrm{div}_H \, \overline{v}) (e^{tM} x_H - x'_H)
  e^{-\tfrac{1}{4} \langle Q_t^{-1} x_H', x_H' \rangle}
  dx_H' \\
  & = e^{tM}(T_H(t) \mathrm{div}_H \, \overline{v})(x_H).
\end{align*}
Here we have used the representation $\mathrm{div}_H \,
\overline{v} = \tr \nabla_H \overline{v}$ and 
with 
a slightly
abusive 
notation $\mathrm{div}_H \, \overline{(T_H(t)v)}$
involves
the actual Ornstein-Uhlenbeck semigroup in
$L^p(\Omega)^2$, while $T_H(t) \mathrm{div}_H \,
\overline{v}$ uses the corresponding semigroup in
$L^p(\R^2)^2$.
Since $e^{tM}$ and $T_H(t)$ are boundedly invertible, one
can conclude form the above identity that
$\mathrm{div}_H \, \overline{(T_H(t)v)} = 0$ if and only if
$\mathrm{div}_H \, \overline{v} = 0$.
Hence, $T_H(t)$ maps onto $L_{\overline{\sigma}}^p(\Omega)$
or, equivalently expressed in terms of projection
$P_p T_H(t) = T_H(t) P_p$.
Therefore, the same holds also for the complementary space
$(\mathds{1}-P_p)T_H(t) = T_H(t)(\mathds{1}-P_p)$.

Concerning $T_z(\cdot)$, we note that $\cL_{z}^0$ is boundedly
invertible if $\Gamma_D \neq \emptyset$.
Since $\cL_z$ is a bounded perturbation, we take
$\lambda > 0$ sufficiently large such that the operator
$(\lambda-\cL_z)$ is boundedly invertible as well as
\begin{align*}
  \mathrm{div}_H \, \overline{(\lambda-\cL_z) v} & =
  \lambda \mathrm{div}_H \, \overline{v} - \mathrm{div}_H
  \, \left( \frac{1}{2h} \int_{-h}^{h} \partial_z^2 v +
  (\mathds{1}-P_p) \tr_D (\partial_z v) \right) \\
  & = \lambda \mathrm{div}_H \, \overline{v} -
  \mathrm{div}_H \, P_p \tr_D (\partial_z v) = \lambda 
  \mathrm{div}_H \, \overline{v}.
\end{align*}
This gives that $\mathrm{div}_H \, \overline{(\lambda -
\cL_z) v} = 0$ if and only if $\mathrm{div}_H \,
\overline{v} = 0$.
Similarly, we have
\begin{align*}
  (\lambda - \cL_z) \nabla_H g = \lambda \nabla_H g
\end{align*}
for $g \in H^{1, p}(\R^2)$.
Hence, the operator $(\lambda-\cL_z)$ is mapping divergence
free fields as well as gradient fields onto. 
Therefore,
the semigroup $T_z$ generated by $\cL_z$ restricts to
$(C_0)$-semigroups on these invariant subspaces;
see e.g. \cite[Theorem~4.5.5]{Pazy}.

In order to verify the commutator relation
\eqref{eq:commuting_resolvents}, we recall that
$\cL_z= \cL^0_z + (\mathds{1}-P_p)(\tr_D \partial_z v)$ and
the arguments above by the Fourier representation
\eqref{eq:FourierSeries}. 
Finally, we conclude that $\cL^0_z$ as well as $\tr_D$
commute with $T_H(t)$ for $t > 0$, and $(\mathds{1}-P_p)$
commutes with $(\cL_H -\lambda)^{-1}$, since $T_H(\cdot)$
defines a $(C_0)$-semigroup on the two dimensional gradient
fields $\Ran (\mathds{1}-P_p)$. 
\end{proof}

Consequently, we appeal to Trotter's results
\cite[Theorem~1]{Trotter1959}, that is, if
$T_A(t) = e^{t A}$ and $T_B(s) = e^{s B}$ commute for
all $s, t > 0$, then the closure $\overline{A + B}$
generates the semigroup defined by $T(t) = e^{tA} e^{tB}$.

\begin{lemma}\label{lemma:Trotter}
  The sum $\cL' = \cL_H + \cL_z$ with $D(\cL') = D(\cL_H)
  \cap D(\cL_z)$ is closed and then $\cL' = \cL_p$.
  Especially, $\cL_p$ is the generator of the semigroup
  $T(t) := T_H(t) T_z(t)$ in $L^p(\Omega)$ for
  $p \in (1, \infty)$.
\end{lemma}

\begin{proof}
First one proves that $D(\cL_H) \cap D(\cL_z) = D(\cL_p)$.
Note that with $\Delta_H$ defined on $D(\Delta_H) = L_z^p
H_{xy}^{2, p}$ we see
\begin{align*}
  D(\cL_H) & = D(\Delta_H) \cap \{ v \in L^p(\Omega)^2 : 
  M x_H \cdot \nabla_H v \in L^p(\Omega)^2 \}, \\
  D(\cL_p) & = \{ v \in H^{2, p}(\Omega)^2 : \partial_z
  u|_{\Gamma_N} = 0, \, u|_{\Gamma_D} = 0 \} \cap \{ v \in
  L^p(\Omega)^2 : M x_H \cdot \nabla_H v \in L^p(\Omega)^2
  \}.
\end{align*}
Hence, it is sufficient to prove that
\begin{align*}
  L_z^p H_{xy}^{2, p} \cap \{ v \in H_z^{2, p} L_z^p :
  \partial_z u|_{\Gamma_N} = 0, \, u|_{\Gamma_D} = 0 \}
  = \{ v \in H^{2, p}(\Omega)^2 : \partial_z u|_{\Gamma_N}
  = 0, \, u|_{\Gamma_D} = 0 \}.
\end{align*}
The inclusion '$\supset$' certainly holds by
\eqref{eq:anisotropicinclusion}.
Besides, for the converse '$\subset$' one has to prove
for mixed derivatives $\nabla_H \partial_z v \in L^p
(\Omega)^{2 \times 2}$.
Recall $D(\cL_z) = D(\cL^0_z)$ and the Fourier
representation in $\R^2 \times S^1$ to see
\begin{align}\label{eq:FourierRepresentation}
  v(x_H, z) := \int_{\R^2} \sum_{k \in \N_0}
  e^{i \xi_H \cdot x_H} \hat{v}_k (\xi_H) 
  \cos \left( \frac{h k z}{\pi} \right) d \xi_H,
\end{align}
where
\begin{align*}
  \hat{v}_k (\xi_H) = \frac{1}{2 \pi} \int_{\R^2}
  e^{-i \xi_H \cdot x_H} \frac{1}{2h} \int_{-h}^h 
  e^{-i k \frac{\pi}{h}z} v(x_H, z) dz dx_H.
\end{align*}
Taking odd or even parts iteratively, one arrives at
Dirichlet, Neumann and mixed boundary conditions.
So, we can show that there exists a constant $C > 0$
such that
\begin{align*}
  \| \nabla_H \partial_z v \|_{L^p(\Omega)}
  \leq C \| \Delta_H v \|_{L^p(\Omega)} + \| \partial_z^2
  v \|_{L^p(\Omega)},
\end{align*}
which proves the claim.

It is known that Ornstein-Uhlenbeck operators in the whole
space are closed;
see \cite[Proposition 3.2]{Metafune2001} and
\cite[Theorem 4.1]{MetafunePruessetall2002}.
The proofs are based on the local elliptic regularity and
the results by Dore and Venni \cite{DoreVenni1987} on
closedness of commuting operators which uses bounded
imaginary powers, respectively.
To prove that the Ornstein-Uhlenbeck operators $\cL_p$ in
the layer $\Omega$ is closed, we assume that there is a
sequence $(v_n) \in D(\cL_p)$ such that both $(v_n)$ and
$(w_n) := (\cL_p v_n)$ are Cauchy sequences in
$L^p(\Omega)^2$.
Take a smooth partition of unity $\mathds{1}_{[-h, h]}$ as
\begin{align*}
  \phi_u, \phi_b : [-h, h] \rightarrow [0, 1], \quad \phi_u
  + \phi_b \equiv 1, \quad \supp \, \phi_u \subset (-h/2,
  h], \quad \supp \, \phi_b \subset [-h, h/2).
\end{align*}
Put $(v^u_n) := (\phi_u v_n)$, $(v^b_n) := (\phi_b v_n)$,
$(w^u_n) := (\phi_u w_n)$ and $(w^b_n) := (\phi_b w_n)$ 
in $L^p(\Omega)^2$.
We choose the alternative extension $E$ by the odd and
even reflection for Dirichlet and Neumann conditions,
respectively, to extend these to sequences
$((Ev)^u_n) = (E \phi_u v_n)$,
$((Ev)^b_n) = (E \phi_b v_n)$,
$((Ew)^u_n) = (E \phi_u w_n)$ and
$((Ew)^b_n) = (E \phi_b w_n)$ in $L^p(\R^3)^2$ with
supports in $\R^2 \times (-2h, 2h)$.
Note that $E \phi_u w_n$ and $E \phi_b w_n$ are in the
domain of the Ornstein-Uhlenbeck operator on $\R^3$
since the extension by reflexion the extended
sequences satisfy the original boundary conditions as well.
Hence, the closedness can be drawn back to the case in the
whole space.
Using the fact that the traces are bounded with respect to
the graph norm of $\cL_p$ which is stronger than $H^{2, p}$
norm, we guarantee that boundary conditions are preserved.
\end{proof}

Combing Lemma~\ref{lemma:commuting} and
Lemma~\ref{lemma:Trotter}, we state the following
proposition.

\begin{proposition}\label{prop:restrictedsemigroup}
  The restriction of $\cL_p$ to
  $L_{\overline{\sigma}}^p(\Omega)$ is the generator of
  the $(C_0)$-semigroup defined by $T(t) = T_H(t) T_z(t)$
  in $L_{\overline{\sigma}}^p(\Omega)$.
\end{proposition}


\section{Smoothing properties}\label{sec:smooth} 
In order to apply a Fujita-Kato type iteration as in
\cite{HieberKashiwabara2015}, we need some smoothing
properties of $T(t)$.

\begin{proposition}\label{prop:LpLqsmoothing}
  Let $p \in (1, \infty)$, $T > 0$ and $\theta_1, \theta_2
  \in [0, 1]$ with $\theta_1 + \theta_2 \leq 1$.
  There exists a constant $C > 0$ depending only on $T$
  and $p$ such that
  \begin{align}\label{ineq:4.1}
    \| e^{t \cL_p} f \|_{H^{2 \theta_1 + 2 \theta_2, p}
    (\Omega)} \leq C t^{-\theta_1} \| f \|_{H^{2 \theta_2,
    p}(\Omega)} \quad \text{for} \quad t \in (0, T)
  \end{align}
  and $f \in H^{2 \theta_2,p}(\Omega)^2$.
  If in addition $\theta_1 > 0$, then
  \begin{align*}
    \lim_{t \to 0+} t^{\theta_1} \| e^{t \cL} f
    \|_{H^{2 \theta_1 + 2 \theta_2, p}(\Omega)} = 0.
  \end{align*}
\end{proposition}

\begin{proof}
For the first inequality with $\theta_1 = 1$ and
$\theta_2 =0$ we compute
\begin{align*}
  \| e^{t \cL_p} f \|_{H^{2, p}(\Omega)} \leq  C \| \Delta e^{t \cL_p} f \|_{H^{2, p}(\Omega)} \leq 
  C \| \Delta_H
  T_H(t) T_z(t) f \|_{L^p(\Omega)} + C \| \Delta_z T_z(t)
  T_H(t) f \|_{L^p(\Omega)},
\end{align*}
for some $C>0$ and with $\Delta = \Delta_H + \Delta_z$. By analyticity of $T_z(t) = e^{t \cL_z}$, one already has
\begin{align*}
  \| \Delta_z T_z(t) T_H(t) f \|_{L^p(\Omega)} \leq
  C t^{-1} \| T_H(t) f \|_{L^p(\Omega)}
\end{align*}
for $t \in (0, T)$ with some $C > 0$.
Considering $\Delta_H T_H(t) T_z(t) f$, we use the explicit
representation of the Kolmogrov kernel \eqref{eq:TH} to
follow the line of \cite[Proposition~3.4]{HieberSawada2005}
in the case $n = 2$ for $g := T_z(t) f$.
From \ref{eq:partialderivatives1},
\ref{eq:partialderivatives2} and 
\ref{eq:partialderivatives3} we deduce for any scalar
valued function $\psi$
\begin{align*}
  \Delta_H [\psi(e^{tM} x_H)] & = \mathrm{div}_H \,
  \nabla_H [\psi (e^{tM} x_H)] \\
  & = \mathrm{div}_H \, [(e^{tM})^T \nabla_H \psi]
  (e^{tM} x_H) \\
  & = [\mathrm{div}_H \, e^{t(M^T+M)} \, \nabla_H \psi]
  (e^{tM} x_H).
\end{align*}
Putting $A_H := \mathrm{div}_H \, e^{t(M^T+M)} \, 
\nabla_H$, we thus get
\begin{align}
  \notag \Delta_H T_H(t) g(x_H, z) & = \frac{1}{4 \pi
  (\det Q_t)^{1/2}} e^{-tM} \int_{\R^2} (\Delta_H)_{x_H}
  g(e^{tM} x_H - x_H', z) e^{-\frac{1}{4} \langle Q_t^{-1}
  x_H', x_H' \rangle} dx_H' \\
  \label{eq:deltath} & = \frac{1}{4 \pi (\det Q_t)^{1/2}}
  e^{-tM}  \int_{\R^2} (A_H)_{x_H} g(e^{tM} x_H - x_H', z)
  e^{-\frac{1}{4} \langle Q_t^{-1} x_H',x_H' \rangle} dx_H'
  \\
  \notag & = \frac{1}{4 \pi (\det Q_t)^{1/2}} e^{-tM} 
  \int_{\R^2} (A_H)_{x_H'} g(e^{tM} x_H - x_H', z)
  e^{-\frac{1}{4} \langle Q_t^{-1} x_H', x_H' \rangle}
  dx_H' \\
  \notag & = \frac{1}{4 \pi (\det Q_t)^{1/2}} e^{-tM}
  \int_{\R^2} g(e^{tM} x_H - x_H', z) (A_H)_{x_H'}
  e^{-\frac{1}{4}\langle Q_t^{-1} x_H',x_H'\rangle} dx_H'
  \\
  \notag & = \frac{1}{4 \pi (\det Q_t)^{1/2}} e^{-tM}
  \int_{\R^2} g(e^{tM} x_H - x_H', z) e^{-\frac{1}{4}
  \langle Q_t^{-1} x_H',x_H'\rangle} \\
  \notag & \qquad \quad \cdot \frac{-1}{2} \left[ \tr
  \left( e^{t(M^T+M)} Q_t^{-1} \right) - \frac{1}{2}
  \langle e^{t(M^T+M)} Q_t^{-1} x_H',Q_t^{-1} x_H' \rangle
  \right] dx_H'.
\end{align}
We substitute $x_H' = Q_t^{1/2}y_H'$ to give
\begin{align*}
  \Delta_H T_H(t) g(x_H, z)
  & = \frac{-1}{8 \pi} e^{-tM} \int_{\R^2} g(e^{tM} x_H -
  Q_t^{-\frac{1}{2}} y_H', z) e^{-\frac{1}{4} \langle y_H',
  y_H' \rangle} \\
  & \qquad \cdot \left[ \tr \left(e^{t(M^T+M)}
  Q_t^{-1} \right) - \frac{1}{2} \langle e^{t(M^T+M)}
  Q_t^{-1} y_H', y_H' \rangle \right] dy_H' \\
  & = \frac{-1}{8 \pi} e^{-tM} \int_{\R^2} g(
  Q_t^{-\frac{1}{2}} (Q_t^{\frac{1}{2}} e^{tM} x_H - y_H'),
  z) e^{-\frac{1}{4} \langle y_H', y_H'\rangle} \\
  & \qquad \cdot \left[ \tr
  \left( e^{t(M^T+M)} Q_t^{-1} \right) - \frac{1}{2}
  \langle e^{t(M^T+M)} Q_t^{-1} y_H',y_H' \rangle \right]
  dy_H'\\
  & = \left(g(Q_t^{-\frac{1}{2}} \cdot,z) \ast_H
  \Psi \right)(Q_t^{1/2} e^{tM}x_H).
\end{align*}
Here
$$
  \Psi(x_H) := \frac{-1}{8 \pi} e^{-tM} e^{-\frac{1}{4}
  \langle x_H, x_H\rangle} \left[ \tr \left(e^{t(M^T+M)}
  Q_t^{-1}\right) - \frac{1}{2} \langle e^{t(M^T+M)}
  Q_t^{-1}x_H,x_H\rangle \right].
$$
By direct calculation we see
\begin{align*}
  |\Psi(x_H)| \leq \frac{1}{16 \pi} \| e^{t(M^T+M)}
  Q_t^{-1} \| \cdot \| e^{-tM} \| \cdot (2+|x_H|^2)
  e^{-\frac{1}{4}|x_H|^2}.
\end{align*}
and by Young's inequality,
\begin{align*}
  \| \Delta_H T_H(t) g(x_H, z) \|_{L^p(\Omega)}
  & = \| \left( g(Q_t^{-1/2} \cdot, \cdot) \ast_H \Psi
  \right) (Q_t^{1/2} e^{tM} \cdot) \|_{L^p(\Omega)} \\
  & = \det(Q_t^{1/2} e^{tM})^{1/p} \| g(Q_t^{-1/2}
  \cdot, \cdot) \ast_H \Psi \|_{L^p(\Omega)} \\
  & \leq \det(Q_t^{1/2} e^{tM})^{1/p} \| g(Q_t^{-1/2}
  \cdot, \cdot) \|_{L^q(\Omega)} \| Psi \|_{L^r(\Omega)} \\
  & = \det(Q_t^{1/2} e^{tM})^{1/p}
  (\det Q_t^{-1/2})^{1/q} \| g \|_{L^q(\Omega)}
  \| \Psi \|_{L^r(\Omega)} \\
  & \leq C_r e^{t \cdot \tr M / p} (\det Q_t)^{\frac{1}{2p}
  - \frac{1}{2q}} \| e^{t(M^T+M)} Q_t^{-1} \| \cdot \| 
  e^{-tM} \| \cdot \| g \|_{L^q(\Omega)} \\
  & \leq C_r e^{t \cdot \tr M / p} (\det Q_t)^{\frac{1}{2p}
  - \frac{1}{2q}} \| e^{tM} \|^2 \cdot \| e^{-tM} \| \cdot
  \| Q_t^{-1} \| \cdot \| g \|_{L^q(\Omega)}
\end{align*}
for $1/p + 1 = 1/q + 1/r$, where $C_r := \frac{1}{8 \pi}
\| (2 + | \cdot |^2) e^{-\frac{1}{4}|\cdot|^2}
\|_{L^r(\Omega)}$.
Since $\| Q_t \|^{-1} \leq \frac{C}{t}$ for $t > 0$ with
some $C$, we have  
\begin{align*}
  \| \Delta_H T_H(t) g(x_H, z) \|_{L^p(\Omega)} \leq
  \frac{C}{t} \| g \|_{L^p(\Omega)}
\end{align*}
for $t \in (0, T)$ by choosing $q = p$ and $r = 1$.

For the case $\theta_1 = 0$, $\theta_2 = 1$ we obtain as
above by the analyticity of $T_z(t)$
\begin{align}\label{eq:normh2p}
  \| e^{t \cL_p} f \|_{H^{2, p}(\Omega)} \leq \| \Delta_H
  T_H(t) T_z(t) f \|_{L^p(\Omega)} + C \| T_H(t) f
  \|_{H^{2, p}(\Omega)}.
\end{align}
Put $g := T_z(t) f$, we have from (\ref{eq:deltath})
\begin{align*}
  \Delta_H T_H(t) g(x_H,z) & = \frac{1}{4 \pi
  (\det Q_t)^{1/2}} e^{-tM} \int_{\R^2} [(A_H)_{x_H} g]
  (e^{tM}x_H -x_H', z) e^{-\frac{1}{4} \langle Q_t^{-1}
  x_H',x_H'\rangle} dx_H' \\
  & = (A_H g \ast_H k_t)(e^{tM}x_H, z).
\end{align*}
Thus, it follows that
\begin{align*}
  \| \Delta_H T_H(t) T_z(t)f \|_{L^p(\Omega)}^p & =
  \int_{-h}^h \| (A_H g \ast_H k_t)(e^{tM}\cdot, z)
  \|_{L^p(\R^2)}^p dz \\
  & = e^{t \cdot \tr M} \int_{-h}^h \| (A_H g \ast_H k_t)
  (\cdot, z) \|_{L^p(\R^2)}^p dz \\
  & \leq e^{t \cdot \tr M} \int_{-h}^h \| A_H g (\cdot, z)
  \|_{L^p(\R^2)}^p \| k_t \|_{L^1(\R^2)}^p dz \\
  & = e^{t \cdot \tr M} \| e^{-tM} \| \int_{-h}^h \| A_H
  g(\cdot, z) \|_{L^p(\R^2)}^p dz \\
  & = e^{t \cdot \tr M} \| e^{-tM} \| \cdot \| A_H g
  \|_{L^p(\Omega)}^p \\
  & \leq e^{t \cdot \tr M} \| e^{-tM} \| \cdot \| 
  e^{t(M^T+M)} \| \cdot \| g \|_{H^{2, p}(\Omega)}^p \\
  & \leq C \| f \|_{H^{2,p}(\Omega)}
\end{align*}
by the analyticity of $T_z(t)$.
The second term in (\ref{eq:normh2p}) can be handled
similarly as
\begin{align*}
  \| T_H(t) f \|_{H^{2, p}(\Omega)} & \leq \| \Delta_H 
  T_H(t) f \|_{L^p(\Omega)} + \| T_H(t) \partial_{zz} f
  \|_{L^{p}(\Omega)} \\
  & \leq C \| f \|_{H^{2, p}(\Omega)} + \| \partial_{zz}
  f \|_{L^{p}(\Omega)} \\
  & \leq C \| f \|_{H^{2, p}(\Omega)}.
\end{align*}
The statement for $\theta_1,\theta_2 \in (0, 1)$ follows
from the interpolation of the inequalities
\begin{align*}
  \| e^{t \cL_p} f \|_{H^{2, p}(\Omega)}
  \leq C t^{-1} \| f \|_{L^p(\Omega)}
  \quad \text{and} \quad 
  \| e^{t \cL_p} f \|_{H^{2, p}(\Omega)}
  \leq C  \| f \|_{H^{2, p}(\Omega)}.
\end{align*}

For the limit $t \to 0+$, we apply to an approximation
argument as $g_n \in C_c(\Omega)$ for $n \in \N$ such that
$g_n \to g$ in $H^{2 \theta_2, p}(\Omega)$.
So, we see
\begin{align*}
  & t^{\theta_1} \| T_H(t) g(x_H,z) \|_{H^{2 \theta_1 + 2
  \theta_2, p}(\Omega)} \\
  & \qquad \leq t^{\theta_1} \| T_H(t) (g(x_H, z) -
  g_n(x_H, z)) \|_{H^{2 \theta_1 + 2 \theta_2, p}(\Omega)}
  + t^{\theta_1} \| T_H(t) g_n(x_H, z) \|_{H^{2 \theta_1 +
  2 \theta_2, p}(\Omega)} \\
  & \qquad \leq C \| g(x_H,z) - g_n(x_H,z) \|_{H^{2
  \theta_2, p}(\Omega)} + t^{\theta_1} \| T_H(t) g_n(x_H,z)
  \|_{H^{2, p}(\Omega)} \\
  & \qquad \leq C \| g(x_H, z) - g_n(x_H, z) \|_{H^{2
  \theta_2, p}(\Omega)} + t^{\theta_1} \| g_n(x_H, z)
  \|_{H^{2, p}(\Omega)}.
\end{align*}
We firstly choose $n$ sufficiently large so that
the first term in the right hand side of the last
inequality small, and secondly take the limit $t \to 0+$.
This completes the proof.
\end{proof}

\begin{remark}
  If $\tr M = 0$ and $M$ has two purely imaginary
  eigenvalues, then the constant $C$ in \eqref{ineq:4.1}
  is independent of $T$.
  Moreover, if $M$ is anti-symmetric, then $C = 1$.
\end{remark}


\section{Mild solutions}\label{sec:mild}
On the primitive equations in the $L^p$ framework with
decaying initial data, time-local unique mild solutions
have been constructed in \cite{HieberKashiwabara2015}
adapting the Fujita-Kato scheme in the spaces
\begin{eqnarray*}
  V_{\theta, p} := [ L^p_{\overline{\sigma}}(\Omega),
  D(A_p) ]_\theta \subset H^{2 \theta, p}(\Omega)^2
  \cap L^p_{\overline{\sigma}}(\Omega)
\end{eqnarray*}
for $p \in (1, \infty)$ with some $\theta \in (0, 1)$,
where $[\cdot, \cdot]_\theta$ denotes by the complex
interpolation;
see \cite[Equation~(4.10)]{HieberKashiwabara2015}.
In what follows, the initial disturbance $v_0$ is taken
from
\begin{align*}
  V_{\frac{1}{p}, p} := \{ v \in H^{\frac{2}{p}, p}
  (\Omega)^2 \cap L^p_{\overline{\sigma}}(\Omega) :
  v|_{\Gamma_D} = 0 \}
\end{align*}
as in \cite[Section~4]{HieberHusseinKashiwabara2016};
where the reader can find the explicit characterization of the interpolation space.
Besides, we deal with mild solutions $v$ in
\begin{align*}
  V_{\frac{1}{2}+\frac{1}{2p}, p} \subset H^{1+\frac{1}{p},
  p}(\Omega)^2 \cap L^p_{\overline{\sigma}}(\Omega) =:
  H_{\overline{\sigma}}^{1+\frac{1}{p}, p}(\Omega)
\end{align*}
for each $t$, whence the mild solution exists.
Here, we appeal to the Fujita-Kato scheme directly in
Sobolev spaces rather than in interpolation spaces,
because the semigroup generated by $\cL_p$ is neither
analytic nor enjoying $L^p$-$L^q$ smoothing properties
onto the operator domain of $\cL_p$.
So, it is 
of 
benefit to 
argue in 
the Sobolev spaces 
$H_{\overline{\sigma}}^{1+\frac{1}{p}, p}(\Omega)$;
see Proposition~\ref{prop:LpLqsmoothing}.

We consider the non-linear and remainder terms
\begin{align}\label{Fp}
  F_p(v) & := - P_p \left( v \cdot \nabla_H v + w
  \partial_z v + 2 M v \right)
\end{align}
rewritten as $F_p(v) = F^*_p(v) - 2 P_p M v$ with
$F^*_p(v) := - P_p (v \cdot \nabla_H v + w \partial_z v)$.
In \cite[Lemma~5.1~(a)]{HieberKashiwabara2015} it has been
derived the estimate
\begin{align}\label{Fp-ast}
  \| F_p^*(v) \|_{H^{1+\frac{1}{p}, p}} \leq C \| v
  \|_{H^{1+\frac{1}{p}, p}}^2
\end{align}
with some $C > 0$ for the case of a bounded cylindrical
domain.
The estimate \eqref{Fp-ast} is still valid in an infinite
layer domain, since by anisotropic H\"older estimates
\begin{align*}
  \| w \partial_z v \|_{L^p(\Omega)} 
  & \leq \| w \|_{L_z^\infty L_{xy}^{2p}}
  \| \partial_z v \|_{L_z^p L_{xy}^{2p}}
  \leq C \| w \|_{H_z^{1, p} L_{xy}^{2p}}
  \| v \|_{H_z^{1,p}L_{xy}^{2p}} \\
  & \leq C \| \mathrm{div}_H \, v \|_{L^p_z L_{xy}^{2p}}
  \| v \|_{H_z^{1, p} H_{xy}^{\frac{1}{p}, p}}
  \leq C \| v \|_{L^p_z H_{xy}^{1, 2p}}
  \| v \|_{H_z^{1, p} H_{xy}^{\frac{1}{p}, p}}
  \leq C \| v \|_{H^{1+\frac{1}{p}, p}}^2
\end{align*} 
with $C > 0$ being a universal constant;
recall that $w = w(v) = - \int_{-h}^z \mathrm{div}_H \, v
(x, y, \xi, t) d\xi$.
Here we have used the Sobolev embeddings
\begin{align*}
  H^{1, p}(-h, h) \hookrightarrow L^{\infty}(-h, h),
  \quad H^{1+\frac{1}{p}, p}(\R^2) \hookrightarrow
  H^{1, 2p}(\R^2), \quad H^{\frac{1}{p}, p}(\R^2)
  \hookrightarrow L^{2p}(\R^2)
\end{align*}
by e.g. \cite[Theorem~3.3.1 and Theorem~2.7.1]{Triebel}
and the Poincar{\'e} inequality $\| w \|_{H^{1, p}_z}
\leq C \| \partial_z w \|_{L^p_z}$ applied to $w$ and
\eqref{eq:anisotropicinclusion}.
Similarly, the estimate for the 
term is derived as
\begin{align*}
  \| v \cdot \nabla_H v \|_{L^p(\Omega)} 
  & \leq \| v \|_{L_z^\infty L_{xy}^{2p}} 
  \| \nabla_H v \|_{L_z^p L_{xy}^{2p}}
  \leq C \| v \|_{H^{1, p}_z L_{xy}^{2p}}
  \| v \|_{L_z^p H_{xy}^{1, 2p}} \\
  & \leq C \| v \|_{H_z^{1,p} H_{xy}^{\frac{1}{p}, p}}
  \| v \|_{L_z^p H_{xy}^{1+\frac{1}{p}, p}}
  \leq C \| v \|_{H^{1+\frac{1}{p}, p}}^2.
\end{align*} 
Since $- 2 P_p M v$ is 
a linear term,
 analogously to
\cite[Lemma~5.1]{HieberKashiwabara2015} we can state the
following lemma.

\begin{lemma}\label{lemma:51temp}
  For $p \in (1, \infty)$, the operator $F_p$ maps
  from $H^{1+\frac{1}{p}}(\Omega)^2$ into
  $L^p_{\overline{\sigma}}(\Omega)$, and there exists a
  constant $C > 0$ such that the following two estimates
  hold:
  \begin{itemize}
  \item[(a)] For $v \in H^{1+\frac{1}{p}}(\Omega)^2$
  \begin{align*}
    \| F_p (v) \|_{L^p_{\overline{\sigma}}(\Omega)}
    \leq C (\| v \|_{H^{1+\frac{1}{p}}(\Omega)^2}^2
    + \| v \|_{H^{1+\frac{1}{p}}(\Omega)^2}).
  \end{align*} 
  \item[(b)] For $v, v_\flat \in H^{1+\frac{1}{p}}
  (\Omega)^2$
  \begin{align*}
    \| F_p (v) - F_p (v_\flat) \|_{L^p_{\overline{\sigma}}
    (\Omega)} \leq C (\| v \|_{H^{1+\frac{1}{p}} (\Omega)}
    + \| v_\flat \|_{H^{1+\frac{1}{p}}(\Omega)} + 1) \| v
    - v_\flat \|_{H^{1+\frac{1}{p}}(\Omega)}.
  \end{align*}
  \end{itemize}
\end{lemma}

Let $T > 0$, and let the space
\begin{eqnarray*}
  S_T := \left\{ v \in C^0([0, T]; V_{\frac{1}{p}, p})
  \cap C^0((0, T]; H_{\overline{\sigma}}^{1+\frac{1}{p}, p}
  (\Omega)^2) : \| v(t) \|_{H^{1+\frac{1}{p}, p}} =
  o(t^{\frac{1}{2p}-\frac{1}{2}}) \, \text{ as } \, t \to
  0 \right\}.
\end{eqnarray*}
This becomes a Banach space equipped with the norm
\begin{align*}
  \| v \|_{S_T} := \sup_{0 \leq t \leq T} \| v(t)
  \|_{H^{\frac{2}{p}, p}} + \sup_{0 \leq t \leq T}
  t^{\frac{1}{2}-\frac{1}{2p}} \| v(t)
  \|_{H^{1+\frac{1}{p}}}.
\end{align*}
The function $v \in C([0, T]; V_{\frac{1}{p}, p})$ is
called a \textit{mild solution} to the primitive equations
with linearly growing data, if $v$ satisfies
\begin{align*}
  v(t) = e^{t \cL_p} v_0 + \int_0^t e^{(t-s)\cL_p} \left(
  P_p f(s) + F_p(v(s))\right) ds
  \quad \text{for} \quad t \in [0, T].
\end{align*}

\begin{theorem}\label{prop:loc_ex}
  Let $p \in (1, \infty)$ and $T > 0$.
  Assume that $v_0 \in V_{\frac{1}{p}, p}$ and $P_p f \in
  C^0((0, T]; L^p_{\overline{\sigma}}(\Omega))$ satisfying
  \begin{eqnarray*}
    \| P_p f(t) \|_{L^p_{\overline{\sigma}}(\Omega)} =
    o(t^{\frac{1}{p}-1}) \quad \hbox{as} \, t \to 0.
  \end{eqnarray*}
  Then there exists $T_\sharp \in (0, T)$ and a unique
  mild solution $v \in S_{T_\sharp}$.
  If in addition $v_0 \in V_{\frac{1}{p}+\varepsilon, p}$
  for some $\varepsilon \in (0, \frac{1}{2}-\frac{1}{2p}]$,
  then
  \begin{align*}
    v \in C^0([0, T_\sharp]; V_{\frac{1}{p} + \varepsilon,
    p}) \cap C^0((0, T_\sharp]; H^{1+\frac{1}{p},
    p}_{\overline{\sigma}}(\Omega)).
  \end{align*}
\end{theorem}

\begin{proof}
The proof is essentially based on that of
\cite[Proposition~5.2]{HieberKashiwabara2015}, using
Proposition~\ref{prop:LpLqsmoothing} and
Lemma~\ref{lemma:51temp}.
We begin with the recursive sequence $(v_m)_{m \in \N}$ 
defined as
\begin{align*}
  v_1(t) := e^{t \cL_p} v_0 + \int_0^t e^{(t-s) \cL_p} P_p
  f(s) ds \quad \hbox{and} \quad
  v_{m+1}(t) := v_1(t) + \int_0^t e^{(t-s) \cL_p} F_p
  (v_m(s)) ds.
\end{align*}
To shorten the notation, put $\gamma := \frac{1}{2} +
\frac{1}{2p} \in (0, 1)$ and $V_\gamma := V_{\gamma, p}$.
As in \cite[Proposition~5.2]{HieberKashiwabara2015} we may
inductively prove that this sequence is well-defined in
$S_T$, and that it converges in this space to prove the
following two properties:
\begin{itemize}
  \item[(a)] There exists $T_\sharp \in (0, T]$ such that
  $\| v_m \|_{S_{T_\sharp}}$ is bounded uniformly in
  $m \in \N$.
  \item[(b)] Let $d_m := v_{m+1} - v_m$, then
  there exists a constant $C_\sharp \in (0, 1)$ such that
  \begin{align*}
    \sup_{0 \leq t \leq T_\sharp} t^{1-\gamma} \| d_{m+1}
    (t) \|_{V_\gamma} \leq C_\sharp \sup_{0 \leq t \leq
    T_\sharp} t^{1-\gamma} \| d_m(t) \|_{V_\gamma}
    \quad \text{for} \quad m \in \N.
  \end{align*}
\end{itemize}
As usual, we 
conclude
that the limit of this successive
approximation is in fact a unique mild solution.
So, our main task is to prove that
\begin{align*}
  k_m(t) := \sup_{0 \leq s \leq t} s^{1-\gamma}
  \| v_m \|_{H^{1+\frac{1}{p}, p}}
\end{align*}
is a bounded sequence uniformly in $m \in \N$ for $t$
sufficiently small.
By definition of $v_1$, we see
\begin{align*}
  \| v_1(t) \|_{H^{1+\frac{1}{p}, p}}
  & \leq \| e^{t \cL_p} v_0 \|_{H^{1+\frac{1}{p}, p}}
  + \int_0^t \| e^{(t-s) \cL_p} \|_{\cL(\lpso,
  H_{\overline{\sigma}}^{1+\frac{1}{p}, p})} 
  \| P_p f(s) \|_{\lpso} ds \\
  & \leq C t^{\gamma-1} \| v_0 \|_{H^{\frac{2}{p}, p}} + C
  \sup_{0 \leq s \leq t} s^{1-\frac{1}{p}} \| P_p f(s)
  \|_{\lpso} \int_0^t (t-s)^{-\gamma} s^{\frac{1}{p} -1} 
  ds \\
  & = C t^{\gamma-1} \| v_0 \|_{H^{\frac{2}{p}, p}} + C
  \sup_{0 \leq s \leq t} s^{1-1/p} \| P_p f(s) \|_{\lpso}
  \cdot t^{1-\gamma} \int_0^1 (1-s)^{-\gamma} s^{1/p -1} ds
\end{align*}
and we thus multiply $t^{1-\gamma}$ in both side to derive
\begin{align*}
  k_1(t) \leq C \| v_0 \|_{H^{\frac{2}{p}, p}} + C
  B\left(\tfrac{1}{2}-\tfrac{1}{2p}, \tfrac{1}{p}\right)
  \sup_{0 \leq s \leq t} \left( s^{1-\frac{1}{p}} \| P_p
  f(s) \|_{L^p_{\overline{\sigma}}(\Omega)}\right)
\end{align*}
where $B(x,y)$ denotes by Euler's beta function.
Here we have used Proposition~\ref{prop:LpLqsmoothing} with
$\theta_1 = \frac{1}{2} - \frac{1}{2p}$, $\theta_2 =
\frac{1}{p}$ to estimate $\| e^{t \cL_p} v_0 \|_{H^{1 +
\frac{1}{p}, p}}$ and $\theta_1 = \frac{1}{2} +
\frac{1}{2p}$, $\theta_2 = 0$ to estimate
$\| e^{(t-s) \cL_p} \|_{\cL(\lpso,
H_{\overline{\sigma}}^{1+\frac{1}{p}, p})}$.
By assumption one can confirm that $k_1(t) \leq k_\flat$
for any small $k_\flat > 0$, if $t$ is taken sufficiently
small.
Similarly, for $m \geq 2$ we have
\begin{align*}
  k_{m+1}(t) \leq k_1(t) + C \sup_{0 \leq s \leq t}
  \left( s^{1-\gamma} \| v_m(s) \|_{H^{1+\frac{1}{p}, p}}
  \right)^2 + C t^{1-\gamma} \sup_{0 \leq s \leq t}
  s^{1-\gamma} \| v_m(s) \|_{H^{1+\frac{1}{p}, p}} 
\end{align*}
with some constant $C > 0$.
We now obtain that
\begin{align*}
  k_{m+1} \leq k_\flat + C \left( k_m(t)^2 + t^{1-\gamma}
  k_m(t) \right).
\end{align*}
Therefore, $k_m(t) \leq 2 k_\flat$ for all $m \in \N$, if
we choose $t$ small enough so that $C t^{1-\gamma} \leq
\frac{1}{4}$ and $k_\flat \leq \frac{1}{8C}$.
Inductively we may check $\lim_{t \to 0+} k_m(t) = 0$.
The other properties 
such as uniqueness
can be
shown by minor adjustments
as in \cite[Section 5]{HieberKashiwabara2015}.
\end{proof}

\begin{remark}
  It is known that the Ornstein-Uhlenbeck semigroup is not
  analytic.
  So, it does not map after short time into its generators
  domain, but only into a Sobolev space.
  Therefore, it is 
  not expected 
  that the mild solution 
  is a strong one.
  However, once we guarantee 
  more
  smoothing on the semigroup, it might be possible to show that the mild
  solution $v$ satisfies \eqref{eq:V} in the classical
  sense as well as $V = v + M x_h$ to \eqref{eq:primequiv}.
\end{remark}



\begin{thebibliography}{99}

\bibitem{Arendt2004}
W.~Arendt.
\newblock Semigroups and evolution equations:
functional calculus, regularity and kernel estimates.
\newblock In {\em Handb. Differ. Equ. Evolutionary
equations. Vol. I}:1--85, 2004.
\newblock \doi{10.1016/s1874-5717(04)80003-3}

\bibitem{BabinMahalovNicolaenko2001}
A.~Babin, A.~Mahalov and B.~Nicolaenko.
\newblock 3D Navier-Stokes and Euler equations with
initial data characterized by uniformly large vorticity.
\newblock {\em Indiana Univ. Math. J.}, 50:1--35, 2001.
\newblock \doi{10.1512/iumj.2001.50.2155}

\bibitem{CampitiGaldiHieber2014}
M.~Campiti, G.~P.~Galdi and M.~Hieber.
\newblock Global existence of strong solutions for 2-dimensional 
Navier-Stokes equations on exterior domains with growing data at infinity.
\newblock {\em Comm. on Pure and Applied Analysis}, 13(4):1613--1627, 2014.
\newblock \doi{10.3934/cpaa.2014.13.1613}

\bibitem{CaoTiti2007}
Ch.~Cao and E.~Titi.
\newblock Global well-posedness of the three-dimensional
viscous primitive equations of large scale ocean and
atmosphere dynamics.
\newblock {\em Annals of Mathematics}, 166:245--267, 2007.
\newblock \doi{10.4007/annals.2007.166.245}

\bibitem{DoreVenni1987}
G.~Dore and A.~Venni.
\newblock On the closedness of the sum of two closed
operators.
\newblock {\em Math. Z.}, 196(2):189--201, 1987.
\newblock \doi{10.1007/BF01163654}

\bibitem{FujitaKato1964}
H.~Fujita and T.~Kato.
\newblock On the Navier-Stokes initial value problem. I.
\newblock {\em Arch. Rational Mech. Anal.},
16:269--315, 1964.
\newblock \doi{10.1007/BF00276188}

\bibitem{GaldiHieberKashiwabara2015}
G.~P.~Galdi, M.~Hieber and T.~Kashiwabara.
\newblock Strong time-periodic solutions to the 3D
primitive equations subject to arbitrary large forces.
\newblock Preprint, \href{https://arxiv.org/abs/1509.02637}{arXiv:1509.02637v1}, 2015.

\bibitem{GallayMaekawa2011}
Th.~Gallay and Y.~Maekawa.
\newblock Three-dimensional stability of Burgers vortices.
\newblock {\em Comm. Math. Phys.}, 302(2):477--511, 2011.
\newblock \doi{10.1007/s00220-010-1132-6}

\bibitem{GigaGriesHusseinHieberKashiwabara2016}
Y.~Giga, M.~Gries, A.~Hussein, M.~Hieber and
T.~Kashiwabara.
\newblock Bounded $H^{\infty}$-calculus for the hydrostatic
Stokes operator on $L^p$-spaces and applications.
\newblock To appear in Proc. Am. Math. Soc.


\bibitem{HallerWiedl2005}
R.~Haller-Dintelmann and J.~Wiedl.
\newblock Kolmogorov kernel estimates for the
Ornstein-Uhlenbeck operator.
\newblock {\em Ann. Sc. Norm. Super. Pisa Cl. Sci. (5)},
4(4):729--748, 2005.
\newblock \url{http://eudml.org/doc/84578}  
  
  
\bibitem{HanShaoWangXu}
B.~Han, S.~Shao, S.~Wang, W.-Q.~Xu.
\newblock Global existence for the 2D Navier-Stokes flow in
the exterior of a moving or rotating obstacle.
\newblock {\em Kinetic and Related Models}, 9(4): 767--776, 2016.  
\newblock \doi{10.3934/krm.2016015} 
  
\bibitem{HieberKashiwabara2015}
M.~Hieber and T.~Kashiwabara.
\newblock Global strong well-posedness of the three
dimensional primitive equations in $L^p$-spaces.
\newblock {\em Arch. Rational Mech. Anal.}, 2016.
\newblock \doi{10.1007/s00205-016-0979-x}
 
\bibitem{HieberHusseinKashiwabara2016}
M.~Hieber, T.~Kashiwabara and A.~Hussein.
\newblock Global strong $L^p$ well-posedness of the 3D
primitive equations with heat and salinity diffusion.
\newblock {\em J. Differential Equations},
261(12):6950--6981, 2016.
\newblock \doi{10.1016/j.jde.2016.09.010}

\bibitem{HieberRhandiSawada2007}
M.~Hieber, A.~Rhandi and O.~Sawada.
\newblock The Navier-Stokes flow for globally Lipschitz
continuous initial data.
\newblock {\em RIMS K{\^{o}}ky{\^{u}}roku Bessatsu},
B1:159--165, 2007. 

\bibitem{HieberSawada2005}
M.~Hieber and O.~Sawada.
\newblock The Navier-Stokes equations in $\R^n$ with
linearly growing initial data.
\newblock {\em Arch. Ration. Mech. Anal.},
175(2):269--285, 2005.
\newblock \doi{10.1007/s00205-004-0347-0}


\bibitem{Hishida1999}
T.~Hishida.
\newblock An existence theorem for the Navier-Stokes
flow in the exterior of a rotating obstacle.
\newblock {\em Arch. Ration. Mech. Anal.},
150(4):307--348, 1999.
\newblock \doi{10.1007/s002050050190}

\bibitem{Kato1984}
T.~Kato.
\newblock Strong $L^p$-solutions of the Navier-Stokes
equation in $\R^m$, with applications to weak solutions.
\newblock {\em Math. Z.}, 187(4):471--480, 1984.
\newblock \url{http://eudml.org/doc/173504}
 
\bibitem{LiTiti2016}
J.~Li and E.~Titi.
\newblock Recent advances concerning certain class of
geophysical flows.
\newblock Preprint \href{https://arxiv.org/abs/1604.01695}{arXiv:1604.01695}, 2016.



\bibitem{Lionsetal1992}
J.~L.~Lions, R.~Temam and Sh.~H.~Wang.
\newblock New formulations of the primitive equations of
atmosphere and applications.
\newblock {\em Nonlinearity}, 5(2):237--288, 1992.
\newblock \doi{10.1088/0951-7715/5/2/001}

\bibitem{Lionsetal1992_b}
J.~L.~Lions, R.~Temam and Sh.~H.~Wang.
\newblock On the equations of the large-scale ocean.
\newblock {\em Nonlinearity}, 5(5):1007--1053, 1992.
\newblock \doi{10.1088/0951-7715/5/5/002}

\bibitem{Lionsetal1993}
J.~L.~Lions, R.~Temam and Sh.~H.~Wang.
\newblock Models for the coupled atmosphere and ocean.
({CAO} {I},{II}).
\newblock {\em Comput. Mech. Adv.}, 1:3--119, 1993.


\bibitem{Majda1986}
A.~Majda.
\newblock Vorticity and the mathematical theory of
incompressible fluid flow.
\newblock {\em Comm. Pure Appl. Math.},
39(S):S187--S220, 1986.
\newblock \doi{10.1002/cpa.3160390711}

\bibitem{Majda2003}
A.~Majda.
\newblock {\em Introduction to PDEs and Waves for the
Atmosphere and Ocean}.
(Courant Lecture Notes in Mathematics vol 9).
\newblock Providence, RI:
American Mathematical Society, 2003.

\bibitem{Metafune2001}
G.~Metafune.
\newblock $L^p$-spectrum of Ornstein-Uhlenbeck operators.
\newblock {\em Ann. Scuola Norm. Sup. Pisa Cl. Sci.},
30(1):97--124, 2001.
\newblock \url{http://eudml.org/doc/84440}

\bibitem{MetafunePruessetall2002}
G.~Metafune, J.~Pr{\"u}ss, A.~Rhandi and R.~Schnaubelt.
\newblock The domain of the Ornstein-Uhlenbeck operator on
an $L^p$-space with invariant measure.
\newblock {\em Ann. Sc. Norm. Super. Pisa Cl. Sci.},
1(2):471--485, 2002.
\newblock \url{http://eudml.org/doc/84478}


\bibitem{Pazy}
A.~Pazy.
\newblock {\em Semigroups of Linear Operators and
Applications to Partial Differential Equations.}
\newblock Springer, New York, 1983.
\newblock \doi{10.1007/978-1-4612-5561-1}  
    


\bibitem{Pedlosky1987}
J.~Pedlosky.
\newblock {\em Geophysical Fluid Dynamics}. Second Edition.
\newblock Springer, New York, 1987.
\newblock \doi{10.1007/978-1-4612-4650-3}    

\bibitem{Triebel}
H.~Triebel.
\newblock {\em Theory of Function Spaces}.
\newblock (Reprint of 1983 edition)
Springer AG, Basel, 2010.
\newblock \doi{10.1007/978-3-0346-0416-1}
 
\bibitem{Trotter1959}
H.~F.~Trotter.
\newblock On the product of semi-groups of operators.
\newblock {\em Proc. Amer. Math. Soc.}, 10:545--551, 1959.
\newblock \doi{10.2307/2033649}

\bibitem{Vallis2006}
G.~K.~Vallis.
\newblock {\em Atmospheric and Oceanic Fluid Dynamics}.
Second Edition.
\newblock Cambridge Univ. Press, 2006.

\bibitem{WashingtonParkinson1986}
W.~M.~Washington and C.~L.~Parkinson.
\newblock {\em An Introduction to Three Dimensional Climate
Modeling}. Second Edition.

\bibitem{Wiedl2006}
J.~Wiedl.
\newblock Analysis of Ornstein-Uhlenbeck operators.
\newblock PhD thesis, TU Darmstadt, 2007.

\end{thebibliography}
\end{document}